\documentclass[reqno,12pt]{amsart}
\usepackage{hyperref}
\usepackage{amsmath,amsthm}
\usepackage{amssymb}
\usepackage{mathabx}
\newtheorem{Theorem}{Theorem } [section]
\newtheorem{lemma}[Theorem]{Lemma}

\numberwithin{equation}{section}

\newcommand{\veee}{{\,\widecheck{\;}\,}}

\newcommand{\veeexi}{{\displaystyle\veee}^{{}_{{\scriptscriptstyle\xi}}}}

\newcommand{\R}{{\mathbb R}}

\font\ff=cmsy10
\def\tiF{\text{\ff F\kern 0pt}{\;}^{ -1}}
\def\tF{\text{\ff F\kern 0pt}}
\begin{document}
\title[]{On the unique continuation property of solutions\\of the  three-dimensional Zakharov-Kuznetsov  equation }
\author{Eddye Bustamante, Jos\'e Jim\'enez Urrea and Jorge Mej\'{\i}a}

\address{Eddye Bustamante M., Jos\'e Jim\'enez U., Jorge Mej\'{\i}a L. \newline
Escuela de Matem\'aticas\\Universidad Nacional de Colombia\newline
A. A. 3840 Medell\'{\i}n, Colombia}
\email{eabusta0@unal.edu.co, jmjimene@unal.edu.co, jemejia@unal.edu.co}

\begin{abstract}
We prove that if the difference of two sufficiently smooth solutions of the three-dimensional Zakharov-Kuznetsov equation
\begin{equation*}
\partial_{t}u+\partial_{x}\triangle u+u\partial_{x}u=0 \text{,}\quad (x,y,z)\in\mathbb R^3, \;t\in[0,1]\;,
\end{equation*}
decays as $e^{-a(x^2+y^2+z^2)^{3/4}}$ at two different times, for some $a>0$ large enough, then both solutions coincide.
\end{abstract}

\maketitle

\section{introduction}
In this paper we consider the  three-dimensional Zakharov-Kuznetsov (ZK) equation
\begin{equation}
\partial_{t}u+\partial_{x}\triangle u+u\partial_{x}u=0 \text{,}\quad (x,y,z)\in\mathbb R^3, \;t\in[0,1]\ ,\label{ZK}
\end{equation}
where $\triangle$ is the three-dimensional spatial Laplace operator.\\
Equation \eqref{ZK} is a three-dimensional generalization of the Korteweg-de Vries (KdV) equation, and was introduced by Zakharov and Kuznetsov in \cite{ZaKu} to describe unidirectional propagation of ionic-acoustic waves in magnetized plasma. The rigorous derivation of the Zakharov-Kuznetsov equation from the Euler-Poisson system with magnetic field was done by Lannes, Linares and Saut in \cite{LLS2013}.

The Cauchy problem associated to the equation \eqref{ZK}, in the context of classical Sobolev spaces, has been recently studied. In \cite{LS2009} Linares and Saut proved that the Cauchy problem is locally well-posed  for initial data in $H^s(\mathbb{R}^3)$ with $s>9/8$, by adapting to the ZK equation the techniques developed in \cite{KT2003} for the Benjamin-Ono equation.  In \cite{RV2012} Ribaud and Vento  obtained the local well-posedness for initial data in $H^s(\mathbb{R}^3)$ for $s>1$. The proof of this result is based on a sharp maximal function estimate in time-weighted spaces. In \cite{MP2014} Molinet and Pilod established the global well-posedness in $H^s(\mathbb{R}^3)$ with $s>1$. In order to obtain the global result the authors work in the framework of the functional spaces introduced by Koch and Tataru in \cite{KT2005}. Other results concerning the well-posedness of the modified ZK equation in three space dimensions can be found in \cite{G2014} and \cite{G2015}.

It is important to point out that the ZK equation in two-space variables has also been studied intensively in the last decades and that many of the results obtained for it have been tried to be generalized to the three-dimensional version. With respect to the Cauchy problem associated to the two-dimensional ZK equation we refer, among others, to the works of Faminskii \cite{F1995}, Linares and Pastor \cite{LP1}, Gr\" unrock and Herr \cite{GH2014}, and Molinet and Pilod \cite{MP2014}.

In addition to the Cauchy problem, another important aspect in the study of certain evolution equations is related to the unique continuation principles. Unique continuation principles consist in the determination of sufficient conditions, in both spatial and time variables, which guarantee that two solutions to the same equation are equal. For nonlinear dispersive equations, unique continuation principles have been obtained for several models including the KdV equation, the nonlinear Schr\"odinger equation and the Benjamin-Ono equation. See for example \cite{B}, \cite{BIM2011}, \cite{EKPV1}, \cite{FP2011}, \cite{KPV}, \cite{KPV2003}, \cite{KPV2006}, \cite{SS} and references therein.

Our aim in this work is to extend to the three-dimensional case the unique continuation principle obtained in \cite{BIM2013} for the two-dimensional ZK equation. In order to achieve this generalization we follow the scheme developed in \cite{BIM2013}, which uses the ideas of \cite{EKPV1} for the KdV equation. This principle is the main result of our work, and we state it as follows:

\begin{Theorem}\label{principal}

Suppose that  for some small $\epsilon>0$ \begin{equation}u_1, u_2\in C([0,1];H^4(\mathbb R^3)\cap L^2((1+x^2+y^2+z^2)^{8/5+\epsilon}\,dx\,dy\,dz))\cap  C^1([0,1];L^2(\mathbb R^3)),\label{hyp}\end{equation} 
are solutions of equation \eqref{ZK}.  Then there exists a universal constant $a_0>0$, such that if for some $a>a_0$
\begin{equation*}
u_1(0)-u_2(0),\, u_1(1)-u_2(1)\in L^2(e^{a(x^2+y^2+z^2)^{3/4}}\,dx\,dy\,dz), \label{condition}\end{equation*}
then
$u_1\equiv u_2$.
 \end{Theorem}
With the same techniques used in  \cite{BJM2016} for the two-dimensional ZK equation,  it can be proved that the Cauchy problem associated to the equation \eqref{ZK} is globally well-posed in the weighted Sobolev space $Z_{4,r}:=H^4(\mathbb R^3)\cap
L^2((1+x^2+y^2+z^2)^{r}\,dx\,dy\,dz)$ for $r\leq 2$.  This shows that  hypothesis \eqref{hyp} in Theorem \ref{principal}   can be changed by this one: $u_1, u_2\in C([0,1];H^4(\mathbb R^3))\cap  C^1([0,1];L^2(\mathbb R^3))$ are solutions of equation \eqref{ZK} such that $u_1(0)$ and $u_2(0)$ belong to the class $L^2((1+x^2+y^2+z^2)^{8/5+\epsilon}\,dx\,dy\,dz)$ for some $0<\epsilon\leq2/5$.

The proof of Theorem \ref{principal} is based upon the comparison of three-dimensional spatial Carleman type estimates and  a lower estimate. To establish the Carleman estimates we used simple procedures which involve partial fraction decompositions and Fourier transform methods.

The Carleman type estimates are stated in the next theorem. 
\begin{Theorem}\label{Carleman}
Let $D:=\mathbb{R}^3\times [0,1]=\{(x,y,z,t)\mid (x,y,z)\in \mathbb{R}^3, t\in[0,1]\}$ and let $w\in C([0,1];H^4(\mathbb{R}^3)\cap C^1([0,1];L^2(\mathbb{R}^3)) $  such  that  for all $t\in[0,1]$ supp$\,w(t)\subseteq K$, where $K$ is a compact subset of $\mathbb R^3$.Then

(i) For $\lambda>0$ and $\beta>0$,
\begin{align}
\|e^{\lambda|x|}e^{\beta|y|}e^{\beta|z|}w\|_{L_t^{\infty}L_{xyz}^2(D)}\leq  &\|e^{\lambda|x|}e^{\beta|y|}e^{\beta|z|}w(0)\|_{L^2(\mathbb{R}^3)}+\|e^{\lambda|x|}e^{\beta|y|}e^{\beta|z|}w(1)\|_{L^2(\mathbb{R}^3)}\notag\\
\notag& +\|e^{\lambda|x|}e^{\beta|y|}e^{\beta|z|}(\partial_t+\partial_x\Delta)w\|_{L_t^1L_{xyz}^2(D)}\;;%\label{carest1}
\end{align}
(ii) There exist $C_1>0$ and $C_2>1$, independent of the set $K$, such that for $\beta\geq1$ and  $\lambda\geq C_2\beta$ 

\begin{align}
\| e^{\lambda|x|}e^{\beta|y|}e^{\beta|z|}Lw\|_{L_x^{\infty}L_{tyz}^2(D)}\leq  
& C_1\lambda^5
\sum_{0\leq k+l+n\leq3}\|e^{\lambda|x|}e^{\beta|y|}e^{\beta|z|}(|\partial_x^k\partial_y^l\partial_z^nw(0)|+|\partial_x^k\partial_y^l\partial_z^nw(1)|)\|_{L^2(\mathbb R^3)}
\notag\\
\notag& +C_1\|e^{\lambda|x|}e^{\beta|y|}e^{\beta|z|}(\partial_t+\partial_x\Delta)w\|_{L_x^1L_{tyz}^2(D)}\;;%\label{carest2}
\end{align}

where $L$ denotes any operator in the set $\{\partial_x,\partial_y,\partial_z,\partial_x^2,\partial_{xy},\partial_{xz},\partial_y^2,\partial_{yz},\partial_z^2\}$.
\end{Theorem}

In order to bound the $H^2$-norm of the difference $v=u_1-u_2$ in the region
$$\{(x,y,z): R-1\leq \sqrt{x^2+y^2+z^2}\leq R\}\times[0,1],$$ with an exponential of the form $e^{-cR^{3/2}}$, we use a lower estimate, which will be established in the following theorem:

\begin{Theorem}\label{Inferior}
Let $u_1, u_2\in C([0,1];H^3(\mathbb{R}^3))\cap C^1([0,1];L^2(\mathbb{R}^3))$ be solutions of \eqref{ZK}. Define $v:=u_1-u_2$. Let $\delta>0$, $r\in(0,\frac12)$, $Q:=\{(x,y,z,t)\mid \sqrt{x^2+y^2+z^2}\leq 1, t\in[r,1-r]\}$ and suppose that $\|v\|_{L^2(Q)}\geq\delta$. For $R>0$ let us define
\begin{align*}
A_R(v):=\big(\int_0^1\int\limits_{R-1\leq\sqrt{x^2+y^2+z^2}\leq R}\sum_{0\leq k+l+n\leq 2}|\partial_x^k\partial_y^l\partial_z^nv(t)(x,y,z)|^2dx\,dy\,dz\,dt\big)^{\frac12}\,.%\label{inf}
\end{align*}
Then there are constants $\overline C=\overline C(r)>0$, $C>0$ and $R_0\geq 2$ such that
\begin{align*}
\|v\|_{L^2(Q)}\leq Ce^{14\overline CR^{\frac32}}A_R(v)\quad\forall R\geq R_0\,.%\label{inf1}
\end{align*}
\end{Theorem}

The central part of the article is the proof of Lemmas \ref{dos} and \ref{EstInf1}, stated below. The proof of Lemma \ref{dos} is an improved version of the proof of the corresponding Lemma in \cite{BIM2013}, while the proof of Lemma \ref{EstInf1}, although it follows the same course of the proof of the corresponding Lemma in \cite{BIM2013}, is technically more complicated due to the change of two to three dimensions.

This article is organized in the same way as \cite{BIM2013}: in section  \ref{DOS} we prove the Carleman estimates stated in Theorem \ref{Carleman}. Section \ref{TRES} contains the proof of Theorem \ref{Inferior}. The proof of the main result, Theorem \ref{principal}, is given in section \ref{CUATRO}.

Throughout the paper the letter $C$ will denote diverse constants, which may change from line to line, and whose dependence on certain parameters is clearly established in all cases. Sometimes, for a parameter $a$, we will use the notation $C_a$ to make emphasis in the fact that the constant $C$ depends upon $a$. We frequently write $f(\cdot_s)$ to denote a function $s\mapsto f(s)$.

\section{Carleman Type Estimates (Proof of Theorem \ref{Carleman})}\label{DOS}

The proof of Theorem \ref{Carleman}, as in \cite{BIM2013}, is based on two lemmas. These lemmas express boundedness properties of the operators $T_0$ and $(\partial_x-\lambda)^k(\partial_y-\beta)^{l}(\partial_z-\beta)^nT_0$, where $k$, $l$ and $n$ are nonnegative  integers with $0< k+l+n\leq 2$, and  $T_0$ is the inverse operator of
\begin{equation*}
g\mapsto e^{\lambda x}e^{\beta y}e^{\beta z}(\partial_t+\partial_x \Delta)e^{-\beta z}e^{-\beta y}e^{-\lambda x}g\,.\label{T0}\end{equation*}
The operator $T_0$ is defined through the space-time Fourier transform by the multiplier
\begin{align*}m_0(\xi,\eta,\zeta,\tau):=\frac1{i\tau+(i\xi-\lambda)^3+(i\xi-\lambda)(i\eta-\beta)^2+(i\xi-\lambda)(i\zeta-\beta)^2}\,.%\label{mult}
\end{align*}

The first lemma, whose proof is analogous to the one of Lemma 2.1 in \cite{BIM2013}, is the following.

\begin{lemma}\label{uno} Let  $h\in L^1(\mathbb{R}^4)$ with $\|h\|_{L_t^1L_{xyz}^2(\mathbb{R}^4)}<\infty$. Then for all $(\lambda,\beta)\ne(0,0)$, $m_0\widehat{h}\in S'(\mathbb {R}^4)$ and $[m_0\widehat{h}]\veee$ defines a bounded function from $\mathbb{R}_t$ with values in $L_{xyz}^2$. Besides,
\begin{align*}\|[m_0\widehat{h}]{\,}\widecheck{\;}\,(\cdot_x,\cdot_y,\cdot_z,t)\|_{L_{xyz}^2(\mathbb{R}^3)}\leq\|h\|_{L_t^1L_{xyz}^2(\mathbb{R}^4)}\quad\forall t\in\mathbb{R}\,,%\label{est1}
\end{align*}
where $\widehat{\;}$  and  $\veee$ denote, respectively, the Fourier transform and its inverse in $S'(\mathbb R^4)$.
\end{lemma}

The second lemma is similar to Lemma 2.2 in \cite{BIM2013}. Its proof also follows the ideas of the proof of Lemma 2.2 in \cite{BIM2013} but it involves some significant changes and simplifications, and by this reason we present it in detail.

\begin{lemma}\label{dos} Let $h\in L^1(\mathbb R^4)$ with $\|h\|_{L_x^1L_{tyz}^2(\mathbb R^4)}<\infty$. For $\beta\geq 1$, $\lambda\in\mathbb R$, $k,l,n\in\{0,1,2\}$, and $0< k+ l+n\leq 2$, let
\[m_{k,l,n}(\xi,\eta,\zeta,\tau):=(i\xi-\lambda)^k(i\eta-\beta)^{l}(i\zeta-\beta)^nm_0(\xi,\eta,\zeta,\tau)\,,\]
where $m_0$ is defined by 
\begin{align*}m_0(\xi,\eta,\zeta,\tau):=\frac1{i\tau+(i\xi-\lambda)^3+(i\xi-\lambda)(i\eta-\beta)^2+(i\xi-\lambda)(i\zeta-\beta)^2}\,.%\label{mult}
\end{align*}
 Then $m_{k,l,n}\widehat h\in S'(\mathbb R^4)$ and there are  constants $C_1>0$ and $C_2>1$, independent from $h$ and $\beta$, such that if $\lambda\geq C_2\beta$ then
\begin{equation*}\|[m_{k,l,n}\widehat h]{\veee}\|_{L_x^{\infty}L_{tyz}^2(\mathbb R^4)}\leq C_1\|h\|_{L_x^1L_{tyz}^2(\mathbb R^4)}\,.\label{mult1}
\end{equation*}
\end{lemma}
\begin{proof} We will consider only the case $k=2$ and $l=n=0$. The proof of the other cases is similar. Let us observe that 
\[m_{2,0,0}(\xi,\eta,\zeta,\tau)=\frac{-iv^2}{v^3+wv-\tau}\,,\]
where $v:=\xi+i\lambda$ and $w:=(\eta+i\beta)^2+(\zeta+i\beta)^2$. If the polynomial $P(v):=v^3+wv-\tau$ has multiple roots then \[\frac{w^3}{\tau^2}+\frac{27}4=0\,.\]
Let us define  the function $T_{\beta}:\mathbb R^3\longrightarrow\mathbb C$ by
\[T_{\beta}(\eta,\zeta,\tau):=\frac{w^3}{\tau^2}+\frac{27}4\,.\]
Let us observe that if $P$ has multiple roots then $T_{\beta}(\eta,\zeta,\tau)=0$ and that
\[T_{\beta}(\eta,\zeta,\tau)=T_{1}(\frac{\eta}{\beta},\frac{\zeta}{\beta},\frac{\tau}{\beta^3})\,.\]
Now we will prove that there exists a compact set $K$, $K\subset\mathbb{R}^3$, such that for $(\eta,\zeta,\tau)\in\mathbb{R}^3-K$ we have
\[|T_1(\eta,\zeta,\tau)|\geq\frac{27}8\,.\]
From an easy calculation it follows that
\begin{align*}|T_1(\eta,\zeta,\tau)|^2&\geq\Big\{\frac{(\eta^2+\zeta^2-2)[(\eta^2+\zeta^2-2)^2-12(\eta+\zeta)^2]}{\tau^2}+\frac{27}4\Big\}^2\\
&=\Big\{\frac{(\eta^2+\zeta^2-2)[(\eta^2+\zeta^2)^2-28(\eta^2+\zeta^2)+12(\eta-\zeta)^2+4]}{\tau^2}+\frac{27}4\Big\}^2\,.
\end{align*}
Let us take $R>0$ such that
\[(\eta^2+\zeta^2-2)\geq\frac12(\eta^2+\zeta^2)\quad\text{and}\quad(\eta^2+\zeta^2)^2-28(\eta^2+\zeta^2)+12(\eta-\zeta)^2+4\geq\frac12(\eta^2+\zeta^2)^2\]
 for $(\eta^2+\zeta^2)>R^2$. Then for $(\eta,\zeta,\tau)\in\mathbb{R}^3$ with $(\eta^2+\zeta^2)>R^2$ we have
 \[|T_1(\eta,\zeta,\tau)|^2\geq\Big\{\frac14\frac{(\eta^2+\zeta^2)^3}{\tau^2}+\frac{27}4\Big\}^2\geq\Big(\frac{27}8\Big)^2\,.\]
 Now suppose  $(\eta^2+\zeta^2)\leq R^2$. There exists $R'>0$ such that if $|\tau|\geq R'$ then
\[\Big|\frac{(\eta^2+\zeta^2-2)[(\eta^2+\zeta^2)^2-28(\eta^2+\zeta^2)+12(\eta-\zeta)^2+4]}{\tau^2}\Big|\leq\frac{27}8\,.\]
In consequence for $(\eta,\zeta,\tau)\in\mathbb{R}^3$ with $(\eta^2+\zeta^2)\leq R^2$ and $|\tau|\geq R'$ we have 
\[|T_1(\eta,\zeta,\tau)|^2\geq\Big(\frac{27}8\Big)^2\,.\] 
We define the compact $K$ to be the set $\{(\eta,\zeta,\tau)\mid (\eta^2+\zeta^2)\leq R^2\;\text{and}\;|\tau|\leq R'\}$. \newline
Now we define the compact set $K_{\beta}$ to be the set $\{(\eta,\zeta,\tau)\mid (\frac{\eta}{\beta},\frac{\zeta}{\beta},\frac{\tau}{\beta^3})\in K\}$. Let $C>1$ such that, for every $(\eta',\zeta',\tau')\in K$, $\max \{|\eta'|, |\zeta'|, |\tau'|\}\leq C$. Then
\begin{equation}\forall (\eta,\zeta,\tau)\in K_{\beta}\quad |\eta|\leq C\beta\,,|\zeta|\leq C\beta\,,|\tau|\leq C\beta^3\,.\label{Kbeta}
\end{equation}
 i) If $(\eta,\zeta,\tau)\in\mathbb R^3-K_{\beta}$  then $|T_{\beta}(\eta,\zeta,\tau)|=|T_{1}(\frac{\eta}{\beta},\frac{\zeta}{\beta},\frac{\tau}{\beta^3})|\geq\frac{27}8$ and in consequence the polynomial $P$ does not have multiple roots, and we can use decomposition in partial fractions to obtain
 \begin{equation}m_{2,0,0}=\sum_{j=1}^3\frac{A_j}{v-v_j}=\sum_{j=1}^3\frac{A_j}{\xi-\operatorname{Re}(v_j)-i[\operatorname{Im}(v_j)-\lambda]}\,,\label{PF}
 \end{equation}
 where $v_j$, $j=1,2,3$, are the different roots of $P$ and
 \begin{equation}A_j:=\lim_{v\to v_j}\frac{(v-v_j)(-iv^2)}{P(v)}=-\frac{iv_j^2}{3v_j^2+w}=-i\big[\frac13-\frac13\frac{w}{3v_j^2+w}\big]\,.\label{Aj}
 \end{equation}
 Now let us see that there is $\delta\in(0,1)$ such that if $(\eta,\zeta,\tau)\in\mathbb R^3-K_{\beta}$ and $v_j\in\mathbb C$ is a root of polynomial $P$ then 
 \begin{equation}|3v_j^2+w|\geq\delta|w|\,.\label{root}\end{equation}
 In other words, let us prove that there is $\delta>0$ such that if $|3v_j^2+w|<\delta|w|$ and $v_j\in\mathbb C$ is a root of polynomial $P$ then $(\eta,\zeta,\tau)\in K_{\beta}$. Taking into account that $v_j$ is a root of $P$ if and only if 
 \begin{equation}\frac1{\frac{v_j^2}{w}\big(\frac{v_j^2}{w}+1\big)^2}=\frac{w^3}{\tau^2}\,,\label{v_j}
 \end{equation}
 let us prove that there is $\delta>0$ such that if $|\frac{v_j^2}{w}+\frac13|<\frac{\delta}3$ and $\frac1{\frac{v_j^2}{w}\big(\frac{v_j^2}{w}+1\big)^2}=\frac{w^3}{\tau^2}$, then $(\eta,\zeta,\tau)\in K_{\beta}$. \newline
 By the continuity of the complex-valued function $\frac1{z(z+1)^2}$ in $z=-\frac13$,  there is $\delta\in(0,1)$ such that if $|z+ \frac13|<\frac{\delta}{3}$ then $|\frac1{z(z+1)^2}+\frac{27}4|<\frac{27}8$. Hence if  $\big|\frac{v_j^2}{w}+\frac13\big|<\frac{\delta}{3}$, we have that 
 \[\Big|\frac1{\frac{v_j^2}{w}\big(\frac{v_j^2}{w}+1\big)^2}+\frac{27}4\Big|<\frac{27}8\,,\]
 and by \eqref{v_j} we conclude that
 \[\Big|\frac{w^3}{\tau^2}+\frac{27}4\Big|<\frac{27}8\,,\]
i.e. $|T_{\beta}(\eta,\zeta,\tau)|<\frac{27}8$ and, in consequence,  $(\eta,\zeta,\tau)\in K_{\beta}$.\newline
 From \eqref{Aj} and \eqref{root} we conclude that if $(\eta,\zeta,\tau)\in\mathbb R^3-K_{\beta}$, then
 \begin{equation}|A_j|\leq\frac13+\frac1{3\delta}\;,\quad j=1,2,3. \label{Aj1}
 \end{equation}
 On the other hand, let us observe that the set $A:=\cup_{j=1}^{3}\{(\eta,\zeta,\tau)\mid \operatorname{Im}(v_j)-\lambda=0\}$ has three-dimensional measure $0$.
 
 For $(\eta,\zeta,\tau)\in(\mathbb R^3-K_{\beta})-A$, $x\in\mathbb R$ and $j=1,2,3$ let us define
 \begin{equation*}
f_j(x)\equiv f_{j,\eta,\zeta,\tau}(x)= \left\{ \begin{array}{ll}
\frac{\sqrt{2\pi}}i \,\chi_{_{_{_{(0,+\infty)}}}}(x) e^{-[\operatorname{Im}(v_j)-\lambda]x} \,\, &\mbox{if $\operatorname{Im}(v_j)-\lambda>0$}, \\ \\
-\frac{\sqrt{2\pi}}i\, \chi_{_{_{_{(-\infty,0)}}}}(x) e^{-[\operatorname{Im}(v_j)-\lambda]x} \,\, &\mbox{if $\operatorname{Im}(v_j)-\lambda<0$}. \end{array} \right. 
\end{equation*}
Then from \eqref{PF} and \eqref{Aj1} it is clear that 
\begin{align}|[m_{2,0,0}(\cdot_{\xi},\eta,\zeta,\tau)]\veeexi(x)|&=\big|\sum_{j=1}^3A_je^{i\operatorname{Re}(v_j)x}f_j(x)\big|\notag\\
&\leq3\sqrt{2\pi}(\frac13+\frac1{3\delta})=\sqrt{2\pi}(1+\frac1{\delta})\,.\label{m20}
\end{align}
ii) If $(\eta,\zeta,\tau)\in K_{\beta}$ , using \eqref{Kbeta} and the fact that $C>1$,  for $\lambda\geq3C\beta$ it follows that 
\begin{align*}|P(v)|=|v^3+wv-\tau|&\geq|\xi+i\lambda|^3-|\eta+i\beta|^2|\xi+i\lambda|-|\zeta+i\beta|^2|\xi+i\lambda|-|\tau|\\
&\geq|\xi+i\lambda|^3-2(C^2+1)\beta^2|\xi+i\lambda|-C\beta^3\\
&\geq|\xi+i\lambda|^3-2(C^2+1)\frac{|\xi+i\lambda|^3}{9C^2}-C\frac{|\xi+i\lambda|^3}{27C^3}\\
&\geq(1-\frac49-\frac1{27})|\xi+i\lambda|^3=\frac{14}{27}|\xi+i\lambda|^3\,.
\end{align*}
Then,
\begin{align*}m_{2,0,0}(\xi,\eta,\zeta,\tau)&=-\frac{i}{\xi+i\lambda}+\big(\frac{-iv^2}{P(v)}+\frac{i}{v}\big)\\
&=-\frac{i}{\xi+i\lambda}+\big(\frac{iwv-i\tau}{vP(v)}\big)\equiv -\frac{i}{\xi+i\lambda}+B(\xi,\eta,\zeta,\tau)\,,
\end{align*}
where
\begin{align*}|B(\xi,\eta,\zeta,\tau)|&\leq\frac{27}{14}\frac1{|\xi+i\lambda|^4}|i(\eta+i\beta)^2(\xi+i\lambda)+i(\zeta+i\beta)^2(\xi+i\lambda)-i\tau|\\
&\leq\frac{27}{14}\frac1{|\xi+i\lambda|^4}2(C^2+1)\beta^2|\xi+i\lambda|+\frac{27}{14}\frac1{|\xi+i\lambda|^4}C\beta^3\\
&\leq\frac{27}{14}\frac1{|\xi+i\lambda|^3}2(C^2+1)\frac{\lambda^2}{9C^2}+\frac{27}{14}\frac1{|\xi+i\lambda|^4}C\frac{\lambda^3}{27C^3}\\
&\leq \frac67\frac{\lambda^2}{|\xi+i\lambda|^3}+\frac1{14}\frac{\lambda^3}{|\xi+i\lambda|^4}\,.
\end{align*}
In consequence,
\begin{align}|[m_{2,0,0}(\cdot_{\xi},\eta,\zeta,\tau)]\veeexi(x)|&\leq\big|\big(\frac{-i}{\xi+i\lambda}\big)\veeexi(x)\big|+|[B(\cdot_{\xi},\eta,\zeta,\tau)]\veeexi(x)|\notag\\
&\leq\sqrt{2\pi}+\big|\frac1{\sqrt{2\pi}}\int_{\mathbb R_\xi}e^{ix\xi}B(\xi,\eta,\zeta,\tau)d\xi\big|\notag\\
\notag&\leq\sqrt{2\pi}+\frac1{\sqrt{2\pi}}\frac67\int_{\mathbb R_\xi}\frac{\lambda^2}{|\xi+i\lambda|^3}d\xi+\frac1{\sqrt{2\pi}}\frac1{14}\int_{\mathbb R_\xi}\frac{\lambda^3}{|\xi+i\lambda|^4}d\xi\\
\notag&\leq\sqrt{2\pi}+\frac1{\sqrt{2\pi}}\frac67(\frac1{\lambda}+2\sqrt2)+\frac1{\sqrt{2\pi}}\frac1{14}(\frac1{\lambda}+\frac43\sqrt2)\\
&\leq\sqrt{2\pi}+\frac5{\sqrt{2\pi}}\leq6\sqrt{2\pi}\,.\label{m20bis}
\end{align}
From \eqref{m20} and \eqref{m20bis} we conclude that for $\lambda\geq3C\beta$ and $\beta\geq1$ and a.e. $(\eta,\zeta,\tau)\in\mathbb R^3$:
\begin{equation}\|[m_{2,0,0}(\cdot_{\xi},\eta,\zeta,\tau)]\veeexi(\cdot_x)\|_{L_x^{\infty}(\mathbb R_x)}\leq(\sqrt{2\pi}(1+\frac1{\delta})+6\sqrt{2\pi})\equiv C_1\label{m20fin}\,.
\end{equation}
Finally, from estimate \eqref{m20fin}, by applying Plancherel's formula in $L_{tyz}^2(\mathbb R^3)$ and Minkowski's integral inequality, it follows that
\[\|[m_{2,0,0}\widehat h]\veee\|_{L_x^\infty L_{tyz}^2(\mathbb R^4)}\leq C_1\|h\|_{L_x^1L_{tyz}^2(\mathbb R^4)}\,.\]
\end{proof}

\text{\sc Remark 1.} The statement of Lemma \ref{dos} is also true if in the definition of the multipliers $m_0$ and $m_{k,l,n}$, instead of $(i\xi-\lambda)$, $(i\eta-\beta)$ and $(i\zeta-\beta)$ we consider the other seven possible triplets of the type $(i\xi\pm\lambda)$, $(i\eta\pm\beta)$ and $(i\zeta\pm\beta)$.

\text{\sc{Proof of Theorem \ref{Carleman}}}

The proof of Theorem \ref{Carleman} follows from Lemmas \ref{uno} and \ref{dos} in a similar way as it was done in the proof of Theorem 1.2 in \cite{BIM2013}.\qed

\section{A lower estimate (proof of Theorem \ref{Inferior})}\label{TRES}

In the proof of Lemma \ref{EstInf1}, basis for the proof of Theorem \ref{Inferior}, we follow the ideas contained in works \cite{I1993} and \cite{EKPV1}. This Lemma refers to a boundedness property of the inverse of the operator associated to the linear part of the  ZK equation. This property is expressed in $L^2$ spaces with an exponential weight $e^\psi$, where the exponent $\psi$ depends on $x,y,z,t$, and a free parameter $\alpha$.

\begin{lemma}\label{EstInf1}
Let $\phi:[0,1] \rightarrow \mathbb R$ be a $C^\infty$ function and let $D:=\mathbb R^3\times[0,1]$. Let us assume that $R>1$ and define
\begin{align*}
\psi(x,y,z,t):=\alpha\left[ (\dfrac xR+\phi(t))^2+\dfrac{y^2}{R^2}+\dfrac{z^2}{R^2}\right]. %\label{defpsi}
\end{align*}

Then, there is $\overline{C}=max(\| \phi'\|_{L^\infty},\| \phi''\|_{L^\infty},1)>0$ such that the inequality
\begin{equation}
\frac{\alpha^{5/2}}{R^3} \|e^\psi g\|_{L^2(D)}+\frac{\alpha^{3/2}}{R^2} \|e^\psi \partial_x g\|_{L^2(D)} \leq \sqrt 2  \|e^\psi (\partial_t+\partial_x\Delta)g\|_{L^2(D)} \label{INF1}
\end{equation}
holds if $\alpha\geq\overline CR^{3/2}$ and  $g\in C([0,1];H^3(\mathbb R^3))\cap C^1([0,1];L^2(\mathbb R^3))$ is such that
\begin{itemize}
\item[(i)] $g(0)=g(1)=0$;
\item[(ii)] there is $M>0$ such that $supp\, g(t)\subset [-M,M]\times[-M,M]\times[-M,M]$ for all $t\in[0,1]$;
\item[(iii)] $supp\,g(\cdot_t)(\cdot_x,\cdot_y,\cdot_z)\subset \{ (x,y,z,t):|\frac{x}{R}+\phi(t)|\geq 1\}$.
\end{itemize}
\end{lemma}
\begin{proof}
Let us define $f:=e^{\psi}g$. Since $e^\psi \partial_x e^{-\psi} f=(\partial_x -\psi_x) f$, to establish (\ref{INF1}) it is sufficient to prove that
\begin{equation}
\frac{\alpha^{5/2}}{R^3} \| f\|_{L^2(D)}+\frac{\alpha^{3/2}}{R^2} \| \partial_x f - \psi_x f\|_{L^2(D)}\leq \sqrt 2 \| Tf\|_{L^2(D)}, \label{INF5}
\end{equation}
where $Tf:=e^\psi (\partial_t+\partial_x\Delta)e^{-\psi}f$.

 Using the fact that $\psi_{xxx}=\psi_{yyx}=\psi_{yx}=\psi_{zzx}=\psi_{zx}=0$, it can be seen that
\begin{align}
\notag Tf=&-\psi_t f +f_t +3\partial_x (-\psi_xf_x)+3\psi_x^2f_x+3\psi_x \psi_{xx} f -\psi_x^3f+f_{xxx}\\
\notag &-\psi_y^2 \psi_x f-\psi_{yy}f_x-2\psi_y f_{xy}-\psi_xf_{yy}+\psi_{yy}\psi_x f+\psi_y^2f_x+2\psi_y\psi_xf_y+f_{xyy}\\
\notag&-\psi_z^2 \psi_x f-\psi_{zz}f_x-2\psi_z f_{xz}-\psi_xf_{zz}+\psi_{zz}\psi_x f+\psi_z^2f_x+2\psi_z\psi_xf_z+f_{xzz}\,.
%\label{INF6}
\end{align}
We write the  operator $T$ as the sum of a symmetric operator $S$ and an antisymmetric operator $A$, as follows:
\begin{align}
\notag S&:=( -3 \partial_x(\psi_x \partial_x \cdotp)-\psi_x^3\cdotp-\psi_t \cdotp)\\
\notag&+(-\psi_y^2\psi_x\cdot-\psi_{yy}\partial_x\cdot-2\psi_y\partial_x\partial_y\cdot-\psi_x\partial_y^2\cdot )\\
\notag&+(-\psi_z^2\psi_x\cdot-\psi_{zz}\partial_x\cdot-2\psi_z\partial_x\partial_z\cdot-\psi_x\partial_z^2\cdot )\\
&\equiv I+II_{y}+II_{z}\, , \label{INF7}\\
\notag A&:=(\partial_x^3 \cdotp + 3 \psi_x^2\partial_x \cdotp + 3\psi_x\psi_{xx} \cdotp+\partial_t \cdotp)\\
\notag& +(\psi_{yy}\psi_x\cdot+\psi_y^2\partial_x\cdot +2\psi_y\psi_x\partial_y\cdot+\partial_y^2\partial_x\cdot)\\
\notag&+(\psi_{zz}\psi_x\cdot+\psi_z^2\partial_x\cdot +2\psi_z\psi_x\partial_z\cdot+\partial_z^2\partial_x\cdot)\\
&\equiv III+IV_y+IV_z\, , \label{INF8}
\end{align}
If we denote by $ \langle \cdotp,\cdotp \rangle$ the inner product in the real Hilbert space $L^2(D)\equiv L^2$, then we have that
\begin{align}
 \| Tf\|_{L^2}^2\geq 2\langle Sf,Af \rangle. \label{INF10}
\end{align}
From \eqref{INF7} and \eqref{INF8} we have
\begin{align}
\notag 2\langle Sf,Af \rangle&=2\langle I f,III f \rangle+2\langle I f,IV_y f \rangle+2\langle I f,IV_z f \rangle\\
\notag&+2\langle II_y f,III f \rangle+2\langle II_y f,IV_y f \rangle+2\langle II_y f,IV_zf \rangle\\
&+2\langle II_z f,III f \rangle+2\langle II_z f,IV_y f \rangle+2\langle II_z f,IV_z f \rangle\label{INF11}\,.
\end{align}
Taking into account $(3.7)$ in \cite{BIM2013} and the symmetry between $y$ and $z$ in the ZK equation it is clear that
\begin{align}
\notag &2\langle I f,III f \rangle+2\langle I f,IV_y f \rangle+2\langle II_y f,III f \rangle+2\langle II_y f,IV_y f \rangle
+2\langle I f,IV_z f \rangle+2\langle II_z f,III f \rangle\\
\notag&+2\langle II_z f,IV_z f \rangle=\\
\notag& \int\limits_D (9\psi_x^4\psi_{xx}-3\psi_{xx}^3+6\psi_{xt}\psi_x^2+\psi_{tt})f^2\\
\notag&+\int\limits_D(6\psi_x^2\psi_y^2 \psi_{xx}+2\psi_{xt}\psi_y^2+4\psi_x^2\psi_y^2\psi_{yy}+\psi_y^4\psi_{xx}
+\psi_{xx}\psi_{yy}^2-6\psi_{xx}^2\psi_{yy} )f^2\\
\notag&+\int\limits_D(6\psi_x^2\psi_z^2 \psi_{xx}+2\psi_{xt}\psi_z^2+4\psi_x^2\psi_z^2\psi_{zz}+\psi_z^4\psi_{xx}
+\psi_{xx}\psi_{zz}^2-6\psi_{xx}^2\psi_{zz} )f^2\\
\notag&+\int\limits_D(18\psi_x^2\psi_{xx}-6\psi_{xt}-6\psi_y^2\psi_{xx}+4\psi_{yy}\psi_y^2-6\psi_z^2\psi_{xx}+4\psi_{zz}\psi_z^2)f_x^2\\
\notag &+\int\limits_D (2\psi_y^2\psi_{xx}-6\psi_x^2\psi_{xx}-2\psi_{xt}+4\psi_x^2\psi_{yy}) f_y^2+\int\limits_D (2\psi_z^2\psi_{xx}-6\psi_x^2\psi_{xx}-2\psi_{xt}+4\psi_x^2\psi_{zz}) f_z^2\\
\notag&+\int\limits_D 24 \psi_x\psi_y\psi_{xx}f_x f_y+\int\limits_D 24 \psi_x\psi_z\psi_{xx}f_x f_z\\
&+\int\limits_D (4\psi_{yy}+6\psi_{xx})f_{xy}^2+\int\limits_D (4\psi_{zz}+6\psi_{xx})f_{xz}^2+\int\limits_D 9\psi_{xx} f_{xx}^2+\int\limits_D \psi_{xx}f_{yy}^2+\int\limits_D \psi_{xx}f_{zz}^2\label{INF24}\,.
\end{align}
In order to have the complete expression for $2\langle Sf,Af \rangle$ we still must calculate $2\langle II_y f,IV_z f \rangle+\langle II_z f,IV_y f \rangle$. For that it is enough to calculate the first term and, by symmetry, the expression for the second one will be obtained from the expression for the the first term by replacing $y$ by $z$ and $z$ by $y$.

Applying integration by parts we obtain
\begin{align}
\notag 2\langle II_y& f,IV_z f \rangle= \int\limits_D(\psi_{xx}\psi_y^2\psi_z^2-\psi_{xx}\psi_{yy}\psi_{zz})f^2+\int\limits_D-\psi_{xx}\psi_z^2f_y^2+\int\limits_D-\psi_{xx}\psi_y^2f_z^2\\
&+\int\limits_D4\psi_x\psi_y\psi_{zz}f_xf_y+\int\limits_D-4\psi_x\psi_{yy}\psi_zf_xf_z+\int\limits_D\psi_{xx}f_{yz}^2+\int\limits_D-8\psi_x\psi_y\psi_zf_{xy}f_z\label{INF25}\,.
\end{align}
By symmetry
\begin{align}
\notag 2\langle II_z& f,IV_y f \rangle= \int\limits_D(\psi_{xx}\psi_z^2\psi_y^2-\psi_{xx}\psi_{zz}\psi_{yy})f^2+\int\limits_D-\psi_{xx}\psi_y^2f_z^2+\int\limits_D-\psi_{xx}\psi_z^2f_y^2\\
&+\int\limits_D4\psi_x\psi_z\psi_{yy}f_xf_z+\int\limits_D-4\psi_x\psi_{zz}\psi_yf_xf_y+\int\limits_D\psi_{xx}f_{yz}^2+\int\limits_D-8\psi_x\psi_z\psi_yf_{xz}f_y\label{INF26}\,.
\end{align}
Taking into account that 
\[\int\limits_D-8\psi_x\psi_y\psi_zf_{xy}f_z+\int\limits_D-8\psi_x\psi_z\psi_yf_{xz}f_y=\int\limits_D8\psi_{xx}\psi_y\psi_zf_yf_z\,,\]
from \eqref{INF25} and \eqref{INF26} follows that
\begin{align}
\notag 2\langle II_y f,IV_z f \rangle+2\langle II_z f,IV_y f \rangle&=\int\limits_D(2\psi_{xx}\psi_y^2\psi_z^2-2\psi_{xx}\psi_{yy}\psi_{zz})f^2+\int\limits_D-2\psi_{xx}\psi_z^2f_y^2\\
&+\int\limits_D-2\psi_{xx}\psi_y^2f_z^2+\int\limits_D2\psi_{xx}f_{yz}^2+\int\limits_D8\psi_{xx}\psi_y\psi_zf_yf_z\label{INF27}\,.
\end{align}
From \eqref{INF11}, \eqref{INF24} and \eqref{INF27} follows that
\begin{align}
\notag &2\langle Sf,Af \rangle=\\
\notag& \int\limits_D (9\psi_x^4\psi_{xx}-3\psi_{xx}^3+6\psi_{xt}\psi_x^2+\psi_{tt})f^2\\
\notag&+\int\limits_D(6\psi_x^2\psi_{xx}(\psi_y^2 +\psi_z^2)+2\psi_{xt}(\psi_y^2+\psi_z^2)+4\psi_x^2(\psi_y^2\psi_{yy}+\psi_z^2\psi_{zz})+\psi_{xx}(\psi_y^4+\psi_{yy}^2+\psi_z^4+\psi_{zz}^2))f^2\\
\notag&\int\limits_D(-6\psi_{xx}^2(\psi_{yy} +\psi_{zz})+2\psi_{xx}\psi_y^2\psi_z^2-2\psi_{xx}\psi_{yy}\psi_{zz})f^2\\
\notag&+\int\limits_D(18\psi_x^2\psi_{xx}-6\psi_{xt}-6\psi_y^2\psi_{xx}+4\psi_{yy}\psi_y^2-6\psi_z^2\psi_{xx}+4\psi_{zz}\psi_z^2)f_x^2\\
\notag &+\int\limits_D (2\psi_y^2\psi_{xx}-6\psi_x^2\psi_{xx}-2\psi_{xt}+4\psi_x^2\psi_{yy}-2\psi_{xx}\psi_z^2) f_y^2\\
\notag&+\int\limits_D (2\psi_z^2\psi_{xx}-6\psi_x^2\psi_{xx}-2\psi_{xt}+4\psi_x^2\psi_{zz}-2\psi_{xx}\psi_y^2) f_z^2\\
\notag&+\int\limits_D 24 \psi_x\psi_y\psi_{xx}f_x f_y+\int\limits_D 24 \psi_x\psi_z\psi_{xx}f_x f_z+\int\limits_D8\psi_{xx}\psi_y\psi_zf_yf_z\\
&+\int\limits_D (4\psi_{yy}+6\psi_{xx})f_{xy}^2+\int\limits_D (4\psi_{zz}+6\psi_{xx})f_{xz}^2+\int\limits_D2\psi_{xx}f_{yz}^2 +\int\limits_D 9\psi_{xx} f_{xx}^2+\int\limits_D \psi_{xx}f_{yy}^2+\int\limits_D \psi_{xx}f_{zz}^2\label{INF28}\,.
\end{align}
Let us observe that $\psi_{xx}=\psi_{yy}=\psi_{zz}=\dfrac{2\alpha}{R^2}\geq 0$. In order to bound $2\langle Sf,Af \rangle$ from below we rewrite the expression in \eqref{INF28} in an adequate manner as follows: 
\begin{align}
\notag &2\langle Sf,Af \rangle=\\
\notag& {\;\;\;}\dfrac{2\alpha}{R^2}\int\limits_D [9\psi_x^4+10\psi_x^2(\psi_y^2+\psi_z^2)+\psi_y^4+\psi_z^4+2\psi_y^2\psi_z^2-12\psi_{xx}^2]f^2\\
\notag &+\dfrac{2\alpha}{R^2}\int\limits_D[(18\psi_x^2-2\psi_y^2-2\psi_z^2)f_x^2+(-2\psi_x^2+2\psi_y^2-2\psi_z^2)f_y^2+(-2\psi_x^2-2\psi_y^2+2\psi_z^2)f_z^2]\\
\notag&+\dfrac{2\alpha}{R^2}\int\limits_D(24\psi_x\psi_yf_xf_y+24\psi_x\psi_zf_xf_z+8\psi_y\psi_zf_yf_z)\\
\notag &+\dfrac{2\alpha}{R^2}\int\limits_D(10f_{xy}^2+10f_{xz}^2+2f_{yz}^2+9f_{xx}^2+f_{yy}^2+f_{zz}^2)\\
\notag&+\int\limits_D[-3\psi_{xx}^3+6\psi_{xt}\psi_x^2+\psi_{tt}+2\psi_{xt}(\psi_y^2+\psi_z^2)]f^2-\int\limits_D6\psi_{xt}f_x^2-\int\limits_D2\psi_{xt}f_y^2-\int\limits_D2\psi_{xt}f_z^2\\
\notag&=\dfrac{2\alpha}{R^2}\int\limits_D\{ [\dfrac{25}{16}\psi_x^4+\dfrac52\psi_x^2(\psi_y^2+\psi_z^2)+\psi_y^4+\psi_z^4+2\psi_y^2\psi_z^2]f^2+4f_{xx}^2+f_{yy}^2+f_{zz}^2\}\\
\notag&+\dfrac{2\alpha}{R^2}\int\limits_D[\dfrac92\psi_x^2(\psi_y^2+\psi_z^2)f^2+6\psi_x\psi_yf_xf_y+6\psi_x\psi_zf_xf_z+2f_{xy}^2+2f_{xz}^2]\\
\notag&+\dfrac{2\alpha}{R^2}\int\limits_D(16\psi_x^2f_x^2+2\psi_y^2f_y^2+2\psi_z^2f_z^2+18\psi_x\psi_yf_xf_y+18\psi_x\psi_zf_xf_z+8\psi_y\psi_zf_yf_z)\\
\notag&+\dfrac{2\alpha}{R^2}\int\limits_D\{[(9-\dfrac{25}{16})\psi_x^4+3\psi_x^2(\psi_y^2+\psi_z^2)-12\psi_{xx}^2]f^2+(2\psi_x^2-2\psi_y^2-2\psi_z^2)f_x^2+(-2\psi_x^2-2\psi_z^2)f_y^2\\
\notag&\quad\quad\quad\quad\quad+(-2\psi_x^2-2\psi_y^2)f_z^2+8f_{xy}^2+8f_{xz}^2+2f_{yz}^2+5f_{xx}^2\}\\
\notag&+\int\limits_D\{[-3\psi_{xx}^3+6\psi_{xt}\psi_x^2+\psi_{tt}+2\psi_{xt}(\psi_y^2+\psi_z^2)]f^2-6\psi_{xt}f_x^2-2\psi_{xt}f_y^2-2\psi_{xt}f_z^2\}\\
&\equiv\dfrac{2\alpha}{R^2}\Big[\int\limits_DIN_1+\int\limits_DIN_2+\int\limits_DIN_3+\int\limits_DIN_4\Big]+\int\limits_DIN_5
\label{INF29}\,.
\end{align}
Now we will bound the integrals $\int\limits_DIN_j$, $j=1,2,3$, from below. Taking into account that
\begin{align}
\notag&\{ [\dfrac{25}{16}\psi_x^4+\dfrac52\psi_x^2(\psi_y^2+\psi_z^2)+\psi_y^4+\psi_z^4+2\psi_y^2\psi_z^2]f^2+4f_{xx}^2+f_{yy}^2+f_{zz}^2\}\\
\notag&=(\dfrac54\psi_x^2f+\psi_y^2f+\psi_z^2f+2f_{xx}+f_{yy}+f_{zz})^2-(5\psi_x^2+4\psi_y^2+4\psi_z^2)ff_{xx}\\
\notag&-(\dfrac52\psi_x^2+2\psi_y^2+2\psi_z^2)ff_{yy}-(\dfrac52\psi_x^2+2\psi_y^2+2\psi_z^2)ff_{zz}-4f_{xx}f_{yy}-4f_{xx}f_{zz}-2f_{yy}f_{zz}\,,
\end{align}
and using integration by parts it follows that
\begin{align}
\notag \int\limits_DIN_1&\geq\int\limits_D[-(5\psi_x^2+4\psi_y^2+4\psi_z^2)ff_{xx}-(\dfrac52\psi_x^2+2\psi_y^2+2\psi_z^2)ff_{yy}-(\dfrac52\psi_x^2+2\psi_y^2+2\psi_z^2)ff_{zz}\\
\notag&\quad\quad-4f_{xx}f_{yy}-4f_{xx}f_{zz}-2f_{yy}f_{zz}]\\
\notag&=\int\limits_D[10\psi_x\psi_{xx}ff_x+5\psi_x^2f_x^2+4(\psi_y^2+\psi_z^2)f_x^2+\dfrac52\psi_x^2f_y^2+4\psi_y\psi_{yy}ff_y+2\psi_y^2f_y^2+2\psi_z^2f_y^2\\
\notag&\quad\quad+\dfrac52\psi_x^2f_z^2+2\psi_y^2f_z^2+4\psi_z\psi_{zz}ff_z+2\psi_z^2f_z^2-4f_{xy}^2-4f_{xz}^2-2f_{yz}^2]\\
\notag&=\int\limits_D[-5\psi_{xx}^2f^2+5\psi_x^2f_x^2+4(\psi_y^2+\psi_z^2)f_x^2+\dfrac52\psi_x^2f_y^2-2\psi_{yy}^2f^2+2\psi_y^2f_y^2+2\psi_z^2f_y^2\\
\notag&\quad\quad+\dfrac52\psi_x^2f_z^2+2\psi_y^2f_z^2-2\psi_{zz}^2f^2+2\psi_z^2f_z^2-4f_{xy}^2-4f_{xz}^2-2f_{yz}^2]\\
\notag&=\int\limits_D[(-5\psi_{xx}^2-2\psi_{yy}^2-2\psi_{zz}^2)f^2+(5\psi_x^2+4\psi_y^2+4\psi_z^2)f_x^2+(\dfrac52\psi_x^2+2\psi_y^2+2\psi_z^2)f_y^2\\
&\quad\quad+(\dfrac52\psi_x^2+2\psi_y^2+2\psi_z^2)f_z^2-4f_{xy}^2-4f_{xz}^2-2f_{yz}^2]\label{INF30}\,.
\end{align}
For the integral $\int\limits_DIN_2$ we have, by applying integration by parts and rewriting in an appropriate manner, that
\begin{align}
\notag& \int\limits_DIN_2=\int\limits_D[\dfrac92\psi_x^2(\psi_y^2+\psi_z^2)f^2+2f_{xy}^2+2f_{xz}^2]+\int\limits_D(6\psi_x\psi_yf_xf_y+6\psi_x\psi_zf_xf_z)\\
\notag&=\int\limits_D[\dfrac92\psi_x^2(\psi_y^2+\psi_z^2)f^2+2f_{xy}^2+2f_{xz}^2]+\int\limits_D(3\psi_{xx}\psi_{yy}f^2-6\psi_x\psi_yff_{xy}+3\psi_{xx}\psi_{zz}f^2-6\psi_x\psi_zff_{xz})\\
\notag&=\int\limits_D(\dfrac92\psi_x^2\psi_y^2f^2-6\psi_x\psi_yff_{xy}+2f_{xy}^2)+(\dfrac92\psi_x^2\psi_z^2f^2-6\psi_x\psi_zff_{xz}+2f_{xz}^2)+3\psi_{xx}(\psi_{yy}+\psi_{zz})f^2\\
\notag&=\int\limits_D[(\dfrac3{\sqrt2}\psi_x\psi_yf-\sqrt2f_{xy})^2+(\dfrac3{\sqrt2}\psi_x\psi_zf-\sqrt2f_{xz})^2+3\psi_{xx}(\psi_{yy}+\psi_{zz})f^2]\\
&\geq\int\limits_D3\psi_{xx}(\psi_{yy}+\psi_{zz})f^2\label{INF31}\,.
\end{align}
With respect to the integral $\int\limits_DIN_3$, taking into account that
\begin{align*}
&(16\psi_x^2f_x^2+2\psi_y^2f_y^2+2\psi_z^2f_z^2+18\psi_x\psi_yf_xf_y+18\psi_x\psi_zf_xf_z+8\psi_y\psi_zf_yf_z)=\\
&(\dfrac92\psi_xf_x+2\psi_yf_y+2\psi_zf_z)^2-\dfrac{17}4\psi_x^2f_x^2-2\psi_y^2f_y^2-2\psi_z^2f_z^2\,,
\end{align*}
we have that
\begin{align}
\int\limits_DIN_3\geq\int\limits_D(-\dfrac{17}4\psi_x^2f_x^2-2\psi_y^2f_y^2-2\psi_z^2f_z^2)\label{INF32}\,.
\end{align}
From \eqref{INF10}, \eqref{INF29}, \eqref{INF30}, \eqref{INF31}, and \eqref{INF32}, since $\psi_{xx}=\psi_{yy}=\psi_{zz}=\dfrac{2\alpha}{R^2}$, it follows that
\begin{align}
\notag \|Tf\|_{L^2}^2&\geq 2\langle Sf,Af \rangle\\
\notag&\geq\dfrac{2\alpha}{R^2}\int\limits_D[(-5\psi_{xx}^2-2\psi_{yy}^2-2\psi_{zz}^2)f^2+(5\psi_x^2+4\psi_y^2+4\psi_z^2)f_x^2+\\
\notag&\quad\quad\quad\quad(\dfrac52\psi_x^2+2\psi_y^2+2\psi_z^2)f_y^2+(\dfrac52\psi_x^2+2\psi_y^2+2\psi_z^2)f_z^2-4f_{xy}^2-4f_{xz}^2-2f_{yz}^2 ]\\
\notag&+\dfrac{2\alpha}{R^2}\int\limits_D3\psi_{xx}(\psi_{yy}+\psi_{zz})f^2\\
\notag&+\dfrac{2\alpha}{R^2}\int\limits_D(-\dfrac{17}4\psi_x^2f_x^2-2\psi_y^2f_y^2-2\psi_z^2f_z^2)\\
\notag&+\dfrac{2\alpha}{R^2}\int\limits_D\{[\dfrac{119}{16}\psi_x^4+3\psi_x^2(\psi_y^2+\psi_z^2)-12\psi_{xx}^2]f^2+(2\psi_x^2-2\psi_y^2-2\psi_z^2)f_x^2\\
\notag&\quad\quad\quad\quad+(-2\psi_x^2-2\psi_z^2)f_y^2+(-2\psi_x^2-2\psi_y^2)f_z^2+8f_{xy}^2+8f_{xz}^2+2f_{yz}^2+5f_{xx}^2\}\\
\notag&+\int\limits_D\{[-3\psi_{xx}^3+6\psi_{xt}\psi_x^2+\psi_{tt}+2\psi_{xt}(\psi_y^2+\psi_z^2)]f^2-6\psi_{xt}f_x^2-2\psi_{xt}f_y^2-2\psi_{xt}f_z^2\}\\
\notag&=\dfrac{2\alpha}{R^2}\int\limits_D\{[\dfrac{119}{16}\psi_x^4+3\psi_x^2(\psi_y^2+\psi_z^2)-15\psi_{xx}^2]f^2+(\dfrac{11}4\psi_x^2+2\psi_y^2+2\psi_z^2)f_x^2\\
\notag&\quad\quad\quad+\dfrac12\psi_x^2f_y^2+\dfrac12\psi_x^2f_z^2+4f_{xy}^2+4f_{xz}^2+5f_{xx}^2\}\\
\notag&+\int\limits_D\{[-3\psi_{xx}^3+6\psi_{xt}\psi_x^2+\psi_{tt}+2\psi_{xt}(\psi_y^2+\psi_z^2)]f^2-6\psi_{xt}f_x^2-2\psi_{xt}f_y^2-2\psi_{xt}f_z^2\}\\
\notag&\geq\int\limits_D\{[\dfrac{119}{16}\psi_x^4\psi_{xx}+3\psi_x^2(\psi_y^2+\psi_z^2)\psi_{xx}-18\psi_{xx}^3+6\psi_{xt}\psi_x^2+\psi_{tt}+2\psi_{xt}(\psi_y^2+\psi_z^2)]f^2\\
&\quad\quad+[(\dfrac{11}4\psi_x^2+2\psi_y^2+2\psi_z^2)\psi_{xx}-6\psi_{xt}]f_x^2+(\dfrac12\psi_x^2\psi_{xx}-2\psi_{xt})(f_y^2+f_z^2)\}\label{INF33}\,.
\end{align}
We observe now that
\begin{align*}
&\psi_x=\dfrac{2\alpha}{R}(\dfrac{x}{R}+\phi(t)); \,\,\,\,\,\,\,\, \psi_y=\dfrac{2\alpha y}{R^2}; \,\,\,\,\,\,\,\,\,\psi_z=\dfrac{2\alpha z}{R^2};\,\,\,\,\,\,\,\, \psi_t=2\alpha (\dfrac{x}{R}+\phi(t))\phi'(t); \\
&\psi_{xx}=\psi_{yy}=\psi_{zz}=\dfrac{2\alpha}{R^2};\,\,\,\,\psi_{tt}=2\alpha(\dfrac{x}{R}+\phi(t))\phi''(t)+2\alpha(\phi'(t))^2; \,\,\,\, \psi_{xt}=\dfrac{2\alpha}{R}\phi'(t);
\end{align*}
and define $\overline C:=\max\{ \| \phi'\|_{L^\infty}, \| \phi''\|_{L^\infty},1\}$. If we take $\alpha\geq\overline CR^{3/2}>1$, for the points $(x,y,z,t)$ where $|\dfrac{x}{R}+\phi(t)|\geq 1$, we can see, since $\max\{\dfrac{\alpha^3}{R^6}, \dfrac{\alpha^3\overline C}{R^3}, \alpha\overline C,\dfrac{\alpha^4}{R^6}\}\leq\dfrac{\alpha^5}{R^6}$ and $\dfrac{\alpha\overline C}{R}\leq\dfrac{\alpha^3}{R^4}$, that
\begin{align}
\notag \dfrac{119}{16}\psi_x^4\psi_{xx}&+3\psi_x^2(\psi_y^2+\psi_z^2)\psi_{xx}-18\psi_{xx}^3+6\psi_{xt}\psi_x^2+\psi_{tt}+2\psi_{xt}(\psi_y^2+\psi_z^2)\\
\notag \geq &\dfrac{238\alpha^5}{R^6}(\dfrac xR+\phi(t))^4+\dfrac{24\alpha^3}{R^4}(\psi_y^2+\psi_z^2)-\dfrac{144\alpha^3}{R^6}-\dfrac{48\alpha^3}{R^3}\overline C(\dfrac xR+\phi(t))^2-2\alpha \overline C|\dfrac xR+\phi(t)|\\
\notag &-\dfrac{4\alpha}R \overline C(\psi_y^2+\psi_z^2)-\dfrac{\alpha^3}{R^4}(\psi_{xx}+\psi_x^2)+\dfrac{\alpha^3}{R^4}(\psi_{xx}+\psi_x^2)\\
\notag \geq&\dfrac{238\alpha^5}{R^6}(\dfrac xR+\phi(t))^4+\dfrac{24\alpha^3}{R^4}(\psi_y^2+\psi_z^2)-\dfrac{144\alpha^5}{R^6}(\dfrac xR+\phi(t))^4-\dfrac{48\alpha^5}{R^6}(\dfrac xR+\phi(t))^4\\
\notag&-2\dfrac{\alpha^5}{R^6}(\dfrac xR+\phi(t))^4-\dfrac{4\alpha^3}{R^4} (\psi_y^2+\psi_z^2)-6\dfrac{\alpha^5}{R^6}(\dfrac xR+\phi(t))^4+\dfrac{\alpha^3}{R^4}(\psi_{xx}+\psi_x^2)\\
\geq &\dfrac{38\alpha^5}{R^6}(\dfrac xR+\phi(t))^4+\dfrac{\alpha^3}{R^4}(\psi_{xx}+\psi_x^2)\geq\dfrac{38\alpha^5}{R^6}+\dfrac{\alpha^3}{R^4}(\psi_{xx}+\psi_x^2)\,,
\label{INF34}
\end{align}

%\end{document}

\begin{align}
\notag(\dfrac{11}4\psi_x^2&+2\psi_y^2+2\psi_z^2)\psi_{xx}-6\psi_{xt}\geq \dfrac{11}4\psi_x^2\psi_{xx}-6\psi_{xt}\\
&\geq\dfrac{22\alpha^3}{R^4}(\dfrac xR+\phi(t))^2-\dfrac{12\alpha}R\overline C\geq \dfrac{22\alpha^3}{R^4}(\dfrac xR+\phi(t))^2-\dfrac{12\alpha^3}{R^4}\geq \dfrac{10\alpha^3}{R^4}, \label{INF35}
\end{align}
\begin{align}
\notag(\dfrac12\psi_x^2\psi_{xx}-2\psi_{xt})&\geq \dfrac{4\alpha^3}{R^4}(\dfrac xR+\phi(t))^2-\dfrac{4\alpha}{R}\overline C\\
&\geq \dfrac{4\alpha^3}{R^4}(\dfrac xR+\phi(t))^2-\dfrac{4\alpha^3}{R^4}(\dfrac xR+\phi(t))^2=0.\label{INF36}
\end{align}
Therefore, since the functions $f$ and $f_x$ are supported in $\{ (x,y,z,t): |\dfrac{x}{R}+\phi(t)|\geq 1\}$, from \eqref{INF33} to \eqref{INF36}, it follows, for $\alpha\geq\overline CR^{3/2}$, that
\begin{align*}
\| Tf\|_{L^2(D)}^2 \geq &\int\limits_D [38\dfrac{\alpha^5}{R^6}+\dfrac{\alpha^3}{R^4}(\psi_{xx}+\psi_x^2)]f^2+\int\limits_D \dfrac{10\alpha^3}{R^4}f_x^2\\
\geq & \dfrac{\alpha^5}{R^6}\int\limits_D f^2+\dfrac{\alpha^3}{R^4} \int\limits_D [(\psi_{xx}+\psi_x^2)f^2+f_x^2]=\dfrac{\alpha^5}{R^6}\| f\|^2+\dfrac{\alpha^3}{R^4}\| f_x -\psi_x f \|^2,
\end{align*}
which proves (\ref{INF5}).
\end{proof}

\text{\sc{Proof of Theorem \ref{Inferior}}}

The proof of this theorem is analogous to the Theorem 1.3 in \cite{BIM2013} and is based upon Lemma \ref{EstInf1}, and in consequence we omit it.\qed

\section{Proof of Theorem \ref{principal}}\label{CUATRO}

We will follow the same lines of the proof of Theorem 1.1 in \cite{BIM2013}.
Let $v:=u_1-u_2$. As it was shown in \cite{BIM2013} it is sufficient to prove that if $v(0), v(1)\in L^2(e^{a(x^2+y^2+z^2)^{\frac34}}dx\,dy\,dz)$ for some $a$ greater than a certain universal positive constant $a_0$ then $v\equiv 0$ in the set $D_0:=\{(x,y,z,t)\mid (x,y,z)\in\mathbb{R}^3, t\in[\frac13,\frac23]\}$. 

Reasoning by contradiction, suppose that $v$ is not identically zero in $D_0$. Without loss of generality (making a translation in the variables $x$, $y$ and $z$ if necessary), we can affirm that if
\begin{equation*}
Q:=\{(x,y,z,t)\mid \sqrt{x^2+y^2+z^2}\leq 1,\;t\in[\frac13,\frac23]\},\end{equation*}
then $\|v\|_{L^2(Q)}>0$.
By Theorem \ref{Inferior} there exist $\overline C>0$, $C>0$,  and $R_0\geq 2$ such that
\begin{equation}
\|v\|_{L^2(Q)}\leq C e^{14\overline C R^{3/2}}A_R(v)\quad \text{for }R\geq R_0.\label {14}\end{equation}
For $R>3$ we take $N\in\mathbb N$, with $N>4C_2 R$, where $C_2$ is the constant in the statement of Theorem \ref{Carleman}. We construct a  $C^\infty$ truncation function
$\phi_{R,N}:  [0,\infty)\longrightarrow \mathbb R$  supported in $(R, N+1)$,  by taking a nonnegative function $\eta\in C^\infty(\mathbb R)$, with unitary integral and $\text{supp}\, \eta\subset(0,1)$,  and defining $\phi_{R,N}(s):=\int_0^s[\eta(s'-{\scriptstyle{R}})-\eta(s'-{\scriptstyle{N}})]\,ds'$. Notice that $0\leq\phi_{R,N}\leq 1$, $\phi_{R,N}\equiv 1$ in $[R+1,N]$ and ${supp}\, \phi_{R,N}\subset (R,N+1)$.
Let us define $\Phi\equiv\Phi_{R,N}:\mathbb R^3\longrightarrow\mathbb R$ by $\Phi(x,y,z)\equiv\Phi_{R,N}(x,y,z):=\phi_{R,N}(\sqrt{x^2+y^2+z^2})$
and
\begin{equation*}
w(t)(x,y,z):=w_{R,N}(t)(x,y,z):=\Phi_{R,N}(x,y,z)v(t)(x,y,z).
\end{equation*}
The function $w$ satisfies the hypotheses of Theorem \ref{Carleman}. Therefore, for $\beta\geq1$ and $\lambda\geq C_2\beta$
\begin{align}
 \|e^{\lambda|x|}e^{\beta|y|}e^{\beta|z|}w\|_{L^\infty_tL_{xyz}^2(D)}+&\sum_{0<k+l+n\leq 2}\|e^{\lambda|x|}e^{\beta|y|}e^{\beta|z|}\partial_x^k\partial_y^l\partial_z^nw\|_{L^\infty_xL^2_{tyz}(D)}\notag\\
 &
 \leq C_1\lambda^5 \sum_{0\leq k+l+n\leq 3}\|e^{\lambda|x|}e^{\beta|y|}e^{\beta|z|}(| \partial_x^k\partial_y^l \partial_z^nw(0)|+| \partial_x^k\partial_y^l \partial_z^nw(1)|)\|_{L^2(\mathbb R^3)}\notag\\
 &+C_1\|e^{\lambda|x|}e^{\beta|y|}e^{\beta|z|}(\partial_t+\partial_x\Delta)w\|_{L_t^1L^2_{xyz}(D)\cap L^1_xL^2_{tyz}(D)}.
\label{quince}
  \end{align}
  But
  \begin{equation}
  (\partial_t+\partial_x\Delta) w=-u_1\partial_xw-(\partial_xu_2)w+F,\label{dieciseis}
  \end{equation}
  where
  \begin{equation}
  F:=(\partial_x\Phi)u_1v+
(\partial_x\Delta\Phi)v+2\nabla\partial_x\Phi\cdot\nabla v+\Delta\Phi\partial_xv+2\nabla\Phi\cdot\nabla\partial_xv+(\partial_x\Phi)\Delta v.\label{diecisiete}  \end{equation}
 Therefore, from \eqref{quince} and \eqref{dieciseis}, and using the facts that $\|\cdot\|_{L^2(D)}\leq \|\cdot\|_{L^\infty_tL^2_{xyz}(D)}$ and $\|\cdot\|_{L^1_tL^2_{xyz}(D)}\leq C\|\cdot\|_{L^2(D)}$ it follows that
 \begin{align}
 \|e^{\lambda|x|}&e^{\beta|y|}e^{\beta|z|}w\|_{L^2(D)}+\sum_{0<k+l+n\leq 2}\|e^{\lambda|x|}e^{\beta|y|}e^{\beta|z|}\partial_x^k\partial_y^l\partial_z^nw\|_{L^\infty_xL^2_{tyz}(D)}\notag\\
 &\leq CC_1\|e^{\lambda|x|}e^{\beta|y|}e^{\beta|z|} u_1\partial_xw\|_{L^2(D)\cap L^1_xL^2_{tyz}(D)} +CC_1\|e^{\lambda|x|}e^{\beta|y|}e^{\beta|z|}(\partial_xu_2)w\|_{L^2(D)\cap L^1_xL^2_{tyz}(D)}\quad(I)\notag\\
 &+CC_1\|e^{\lambda|x|}e^{\beta|y|}e^{\beta|z|} F\|_{L^2(D)\cap L^1_xL^2_{tyz}(D)}\quad(II)\notag\\
 &+C_1\lambda^5\sum_{0\leq k+l+n\leq 3}\|e^{\lambda|x|}e^{\beta|y|}e^{\beta|z|} (|\partial_x^k\partial_y^l\partial_z^nw(0)|+|\partial_x^k\partial_y^l\partial_z^n(w(1)|)\|_{L^2(\mathbb R^3)}\quad(III)\notag\\
 & =I+II+III.
 \label{dieciocho}
\end{align} 
 We now estimate $I$, $II$, and $III$ separately.
 
 {\textit{Estimation of $I$:}
By observing that  $w$ is supported in the set \begin{equation}D_R:=\{(x,y,z)\mid x^2+y^2+z^2\geq R^2\}\times[0,1]\label{DR},\end{equation} and applying H\"older's inequality we see that
\begin{align}
I&\leq C\|u_1\|_{L^2_xL^\infty_{tyz}(D_R)}\|e^{\lambda|x|+\beta|y|+\beta|z|} \partial_xw\|_{L_x^\infty L_{tyz}^2(D)}+ C\|u_1\|_{L^1_xL^\infty_{tyz}(D_R)}\|e^{\lambda|x|+\beta|y|+\beta|z|}\partial_x w\|_{L^\infty_xL^2_{tyz}(D)}\notag\\
&+C\|\partial_xu_2\|_{L^\infty({D_R})}\|e^{\lambda|x|+\beta|y|+\beta|z|} w\|_{L^2(D)}+ C\|\partial_xu_2\|_{L^2_xL^\infty_{tyz}(D_R)}\|e^{\lambda|x|+\beta|y|+\beta|z|}w\|_{L^2(D)}.\label{dieciochoa}
\end{align} 
We will now see that the norms involving $u_1$ and $\partial_xu_2$ on the right hand side of \eqref{dieciochoa} tend to zero as $R\to\infty$. For that we use the following interpolation Lemma (see \cite{NP2009}).
\begin{lemma}\label{tres} Let $a, b>0$. Assume that $J^af:=(1-\Delta)^{\frac{a}2}f\in L^2(\mathbb{R}^n)$ and $(1+\sum_{i=1}^nx_i^2)^{\frac{b}2}f\in L^2(\mathbb{R}^n)$, where $x=(x_1,\cdots,x_n)\in \mathbb{R}^n$. Then for any $\theta\in(0,1)$
\begin{align}
\|J^{\theta a}((1+\sum_{i=1}^nx_i^2)^{(1-\theta)\frac{b}2}f)\|_{L^2(\mathbb{R}^n)}\leq C\|(1+\sum_{i=1}^nx_i^2)^{\frac{b}2}f\|^{(1-\theta)}_{L^2(\mathbb{R}^n)}\|J^af\|_{L^2(\mathbb{R}^n)}^{\theta}\,.\label{interpo}
\end{align}
\end{lemma}

Applying \eqref{interpo} with $a=4$, $b=8/5+\epsilon$ and $\theta=3/8+\epsilon/8$ (with $\epsilon$ as in \eqref{hyp}) we have that
\begin{equation}
\|J^{ 3/2+\epsilon/2}((1+x^2+y^2+z^2)^{\frac12(1+\epsilon_1)} f)\|_{L^2(\R^3)}\leq C\|(1+x^2+y^2+z^2)^{\frac12(\frac85+\epsilon)}f\|_{L^2(\R^3)}^{1-\theta}\,\|J^4f\|_{L^2(\R^3)}^{\theta},\label{inter1}\end{equation}
where $\epsilon_1=\frac{17}{40}\epsilon-\frac18\epsilon^2>0$.
Applying \eqref{inter1} with $f=u_j(t)$, $j=1,2$, from \eqref{hyp} and from the embbeding of $H^{3/2+\epsilon/2}(\mathbb R^3)$ in $L^\infty(\mathbb R^3)\cap C(\mathbb R^3)$, we conclude that
 \begin{equation}|u_j(t)(x,y,z)|\leq \frac{C}{(1+x^2+y^2+z^2)^{\frac12(1+\epsilon_1)}}, \quad\text{for all }(x,y,z,t)\in D.\label{dec1}\end{equation}
Since $3/2+\epsilon/2>1$, \eqref{inter1} is also valid for $f=u_2(t)$ with $J^1$ instead of $J^{ 3/2+\epsilon/2}$, and we can apply the product rule for derivatives to obtain that $\|(1+x^2+y^2+z^2)^{\frac12(1+\epsilon_1)}\partial_xu_2(t)\|_{L^2(\mathbb R^3)}$ is a bounded function of $t\in[0,1]$. Thus, applying \eqref{interpo} with $f=\partial_xu_2(t)$, $a=3$, $b=1+\epsilon_1$ and $\theta=1/2+\epsilon_2$ with $\epsilon_2>0$ small, we can conclude that
\begin{align*}
\|J^{ 3/2+3\epsilon_2}((1+x^2+y^2+z^2)^{(1/2-\epsilon_2)\frac12(1+\epsilon_1)} \partial_xu_2(t))\|_{L^2(\R^3)}\\
\leq C\|(1+x^2+y^2+z^2)^{\frac12(1+\epsilon_1)}\partial_xu_2(t)\|_{L^2(\R^3)}^{(1/2-\epsilon_2)}\,\|J^3\partial_xu_2(t)\|_{L^2(\R^3)}^{(1/2+\epsilon_2)}\leq C\,,
\end{align*}
and, in consequence, using the embedding of $H^{3/2+3\epsilon_2}(\mathbb R^3)$ in $L^\infty(\mathbb R^3)\cap C(\mathbb R^3)$, we have
\begin{equation}|\partial_xu_2(t)(x,y,z)|\leq \frac{C}{(1+x^2+y^2)^{\frac12(1+\epsilon_1)(1/2-\epsilon_2)}}, \quad\text{for all }(x,y,z,t)\in D.\label{dec2}\end{equation}
From the decay properties   expressed in \eqref{dec1} and \eqref{dec2} is now easy to see that the four norms in \eqref{dieciochoa} involving $u_1$ and $\partial_xu_2$ tendo to zero as $R\to \infty$, and therefore the exists 
 $R_1>R_0$ such that for $R>R_1$ the term $I$  can be absorbed by the terms $\|e^{\lambda|x|+\beta|y|+\beta|z|} \partial_xw\|_{L^2(D)}$ and $\|e^{\lambda|x|+\beta|y|+\beta|z|} w\|_{L^\infty_xL^2_{tyz}(D)}$ on the left hand side of \eqref{dieciocho}.

{\textit{Estimation of $II$:}

To estimate $II$ we use the expression for $F$ given in \eqref{diecisiete}  and observe that the derivatives of $\Phi$ are supported in the sets
$\{(x,y,z)\mid \sqrt{x^2+y^2+z^2}\in(R,R+1)\}$ and $\{(x,y,z)\mid \sqrt{x^2+y^2+z^2}\in(N,N+1)\}$.
For the estimation in the first set we bound $e^{\lambda|x|+\beta|y|+\beta|z|}$ by $e^{(\lambda+2\beta)(R+1)}$ and apply Cauchy-Schwarz inequality. For the estimation in the second set we observe that since $v(0), \,v(1)\in L^2(e^{a(x^2+y^2+z^2)^{3/4}}\,dx\,dy\,dz)$,
%for all $a>0$,
it follows that for all $\overline\lambda>0$ and all $\overline\beta>0$
, $v(0), \,v(1)\in L^2(e^{2\overline\lambda|x|+2\overline\beta|y|+2\overline\beta|z|}\,dx\,dy\,dz)$.
Applying to the equation
\begin{equation}
\partial_tv+\partial_x\Delta v+u_1\partial_xv+(\partial_xu_2)v=0\,,
\end{equation}
a procedure similar to that given in Theorem 1.3  of \cite{BIM2011} for   the ZK equation, we can observe that
$v$ is a bounded function from the time interval $[0,1]$ with values in $H^3(e^{2\overline\lambda|x|+2\overline\beta|y|+2\overline\beta|z|}\,dx\,dy\,dz)$ (i.e. $\partial^{\alpha}v(t)\in L^2(e^{2\overline\lambda|x|+2\overline\beta|y|+2\overline\beta|z|}\,dx\,dy\,dz)$ for all multi-index $\alpha=(\alpha_1,\alpha_2,\alpha_3)$ with $|\alpha|\leq 3$). Thus we can take $\overline\lambda=\lambda+1$ and $\overline{\beta}=\beta+1$ and apply Cauchy-Schwarz inequality. In this manner we obtain
\begin{equation}
II\leq CR^{1/2}e^{(\lambda+2\beta)(R+1)}+C_{\lambda,\beta} N^{1/2}e^{-N} ,\label{dieciochob}\end{equation} 
where $C$ does not depend of $\lambda$, $\beta$, $R$, and $N$, and $C_{\lambda,\beta}$ does not depend of $N$ and $R$.

{\textit{Estimation of $III$:}

Since $\Phi_{R,N}$ and its derivatives are bounded by a constant independent of $R$ and $N$,  it follows that
\begin{align}
\notag III\leq &C\lambda^5\sum_{0\leq k+l+n\leq 3}\|e^{\lambda|x|+\beta|y|+\beta|z|} (|\partial_x^k\partial_y^l\partial_z^nv(0)|+|\partial_x^k\partial_y^l\partial_z^n(v(1)|)\|_{L^2(\{(x,y,z):x^2+y^2+z^2\geq R^2\})}\\
&=:H_R(\lambda,\beta).\label{dieciochoc} \end{align}

We now consider  the region $\Omega_R:=\{(x,y,z,t)\mid\sqrt{x^2+y^2+z^2}\in[6C_2R-1,\, 6C_2R],\,t\in[0,1]\}$, where $C_2$ is the constant in the statement of Theorem \ref{Carleman}. Then, from \eqref{DR}, $\Omega_R\subset D_R$, and since $N>6C_2R$, $v$ and $w$ coincide in $\Omega_R$.    We now return to \eqref{dieciocho}, replace its left hand side by a smaller amount, apply \eqref{dieciochob} and \eqref{dieciochoc}, and make $N\to\infty$ to obtain that
\begin{align}
\notag\|e^{\lambda|x|+\beta|y|+\beta|z|} &v\|_{L^2(\Omega_R)}+ \sum_{0<k+l+n\leq 2}\|e^{\lambda|x|+\beta|y|+\beta|z|}\partial_x^k\partial_y^l\partial_z^nv\|_{L^2(\Omega_R)}\\
&\leq
CRe^{(\lambda+2\beta)(R+1)}+CR^{1/2}H_R(\lambda,\beta)\leq
Ce^{(\lambda+2\beta+1)(R+1)}+CR^{1/2}H_R(\lambda,\beta).\label{veinte}
\end{align}
For $\lambda\geq C_2$, let $\beta:=\lambda/ C_2\geq 1$. Then, for $(x,y,z,t)\in\Omega_R$, 
\begin{equation}
\lambda|x|+\beta|y|+\beta|z|\geq \frac{\lambda}{ C_2}(|x|+|y|+|z|) \geq \frac{\lambda}{ C_2}(6C_2R-1)\geq  \frac{17}3\lambda R.\label{veintea}
\end{equation}
In this way, bearing in mind the definition of $A_{6C_2R}(v)$ given in the statement of Theorem \ref{Inferior} and taking into account that $\lambda\geq C_2$ and $R>3$, from \eqref{veinte} and \eqref{veintea} we conclude that
\begin{equation}
e^{\frac{17}3\lambda R}A_{6C_2R}(v)\leq Ce^ {(\lambda+2\frac{\lambda}{ C_2}+1)(R+1)}+CR^{1/2}H_R(\lambda,\frac{\lambda}{ C_2})
\leq Ce^ {\frac{16}3\lambda R}+R^{1/2}H_R(\lambda,\frac{\lambda}{ C_2}),\label{veinticuatro}
\end{equation}
We now take $\lambda=\frac1{10} aR^{1/2}$ with $R$ large enough. If $x^2+y^2+z^2\geq R^2$ we see that
\begin{align}
\notag \lambda|x|+\beta|y|+\beta|z|&=\lambda(|x|+\frac{|y|}{C_2}+\frac{|z|}{C_2})=\frac1{10}aR^{1/2}(|x|+\frac{|y|}{C_2}+\frac{|z|}{C_2})\\
&\leq \frac1{10}a(x^2+y^2+z^2)^{1/4}(|x|+\frac{|y|}{C_2}+\frac{|z|}{C_2})\notag\\
&\leq\frac1{10}a(x^2+y^2+z^2)^{1/4}2(x^2+\frac{y^2}{C_2^2}+\frac{z^2}{C_2^2})
\leq\frac1{5} a(x^2+y^2+z^2)^{3/4}.\label{veintiuno}\end{align}
Since $v(0), v(1)\in L^2(e^{a(x^2+y^2+z^2)^{3/4}}\,dx\,dy\,dz)\cap C([0,1]; H^4(\R^3))$, by an interpolation argument which can be proved using smooth truncation functions and integration by parts, it can be seen that if $0\leq k+l+n\leq 3$, then $\partial_x^k\partial_y^l\partial_zv(0)$ and $\partial_x^k\partial_y^l\partial_zv(1)$  belong to the class $L^2(e^{\frac14a(x^2+y^2+z^2)^{3/4}}\,dx\,dy\,dz)$  In this way, from \eqref{veintiuno} and the definition of $H_R(\lambda,\beta)$ given in \eqref{dieciochoc} we obtain that
\begin{align*}
&R^{1/2}H_R(\lambda,\frac{\lambda}{ C_2})\\
&\leq CR^{1/2}a^5R^{5/2}\sum_{0\leq k+l+n\leq 3}\|e^{\frac{1}5a(x^2+y^2+z^2)^{3/4}}(|\partial_x^k\partial_y^l\partial_z^nv(0)|+|\partial_x^k\partial_y^l\partial_z^nv(1)|)\,\|_{L^2(\{(x,y,z):x^2+y^2+z^2\geq R^2\})}\\
&\leq Ca^5\sum_{0\leq k+l+n\leq 3}\|R^3e^{\frac15 a(x^2+y^2+z^2)^{3/4}}(|\partial_x^k\partial_y^l\partial_z^nv(0)|+|\partial_x^k\partial_y^l\partial_z^nv(1)|)\,\|_{L^2(\{(x,y,z):x^2+y^2+z^2\geq R^2\})}\\
&\leq Ca^5\sum_{0\leq k+l+n\leq 3}\|e^{\frac14 a(x^2+y^2+z^2)^{3/4}}(|\partial_x^k\partial_y^l\partial_z^nv(0)|+|\partial_x^k\partial_y^l\partial_z^nv(1)|)\,\|_{L^2(\R^3)}\equiv C_a.
\end{align*}
Hence, from \eqref{veinticuatro} with $\lambda =\frac1{10}aR^{1/2}$ we have that
\begin{equation*}
e^{\frac{17}{30}a R^{3/2}}A_{6C_2R}(v)\leq Ce^ {\frac{16}{30}a R^{3/2}}+C_a\leq C_ae^ {\frac{16}{30}aR^{3/2}},
\end{equation*}
and thus
\begin{equation*}
A_{6C_2R}(v)\leq C_ae^ {-\frac1{30}aR^{3/2}},
\end{equation*}
In this way, from \eqref{14} it follows that for $R$ large enough
\[\|v\|_{L^2(Q)}\leq C e^{14\overline C (6C_2R)^{3/2}}A_{6C_2R}(v) \leq C e^{14(6C_2)^{3/2}\overline C R^{3/2}}C_ae^{-\frac1{30}aR^{3/2}}.\]
Thus, if  $a>a_0:=30\times14(6C_2)^{3/2}\overline{C}$, then, by  making $R\to\infty$, we see that $\|v\|_{L^2(Q)}=0$, which contradicts  the original fact that $\|v\|_{L^2(Q)}>0$. Therefore $v\equiv0$, and Theorem \ref{principal} is proved.
\qed

\textbf{Acknowledgments}. Partially supported by Fondo Nacional de Financiamiento Para la Ciencia, la Tecnolog\'ia y la Innovaci\'on Francisco Jos\'e de Caldas, Project ``Ecuaciones Diferenciales Dispersivas y El\'ipticas No Lineales", contrato Colciencias FP44842-087-2015.


\begin{thebibliography}{12}

\bibitem{B} Bourgain, J., \textit{On the compactness of the support of solutions of dispersive Equations,} IMRN,  International Mathematics Research Notices {\bf 9} (1997), 437-444.

\bibitem{BIM2011} Bustamante, E., Isaza, P., Mej\'{\i}a, J.,  \textit{On the support of solutions to the Zakharov–-Kuznetsov equation,} J. Differential Equations {\bf 251} (2011), 2728-2736.

\bibitem{BIM2013} Bustamante, E., Isaza, P., Mej\'ia, J., \textit{On uniqueness properties of solutions of the Zakharov-Kuznetsov equation}, J. Funct. Anal.{\bf  264} (2013), 2529-2549.

\bibitem{BJM2016} Bustamante, E., Jim\'enez, J., Mej\'ia, J., \textit{The Zakharov-Kuznetsov equation in weighted Sobolev spaces}, J. Math. Anal. Appl. {\bf 433} (2016), No 1, 149-175.

\bibitem{EKPV1} Escauriaza, L., Kenig, C., Ponce, G., Vega, L., \textit{On uniqueness properties of solutions of the k-Generalized KdV equations,} J. Funct. Anal. {\bf 244} (2007), 504-535.

\bibitem{F1995} Faminskii, A.V., \textit{The Cauchy problem for the Zakharov-–Kuznetsov equation}, Differential Equations {\bf 31} (6)(1995), 1002-1012.

\bibitem{FP2011} Fonseca, G., Ponce, G., \textit{The IVP for the Benjamin-Ono equation in weighted Sobolev spaces}, J. Funct. Anal.{\bf  260} (2011), 436-459.

\bibitem{G2014} Gr\" unrock, A., \textit{ A remark on the modified Zakharov-Kuznetsov equation in three space dimensions}, Math. Res: Lett. {\bf 21} (2014), No. 1, 127-131.  

\bibitem{G2015} Gr\" unrock, A., \textit{On the generalized Zakharov-Kuznetsov equation at critical regularity}, arXiv:1509.09146. 

\bibitem{GH2014} Gr\" unrock, A., Herr, S., \textit{The Fourier restriction method norm for the Zakharov-Kuznetsov equation}, Discrete and continuous dynamical systems {\bf34} (2014),  No. 5, 2061-2068.

\bibitem{I1993}
Isakov, V.,
\emph{Carleman type estimates in anisotropic case and applications,}
J. Differential Equations \textbf{105}
(1993), 217-238.

\bibitem{KPV} Kenig, C., Ponce, G., Vega, L., \textit{On the support of solutions to the g-KdV equation,} Ann. Inst. H. Poincar\'e Anal. Non lin\'eaire {\bf 19} (2002), 191-208.

\bibitem{KPV2003} Kenig, C., Ponce, G., Vega, L., \textit{On unique continuation for nonlinear Schrodinger equations}, Commun. Pure Appl. Math. {\bf{56}} (2003), 1247-1262.

\bibitem{KPV2006} Kenig, C., Ponce, G., Vega, L., \textit{On uniqueness properties of solutions of Schrodinger equations}, Comm. Partial Diff. Equations, {\bf 31} (2006), 1811-1823.

\bibitem{KT2003} Koch, H., Tzvetkov, N., \textit{ On the local well-posedness of the Benjamin-Ono equation in $H^s(\mathbb{R})$}, Int. Math. Res. Not. {\bf{26}} (2003), 1449-1464.

\bibitem{KT2005} Koch, H., Tataru, D., \textit{Dispersive estimates for the principally normal pseudodifferential operators}, Commun. Pure Appl. Math. {\bf{58}} (2005), 217-284.

\bibitem{LLS2013} Lannes, D., Linares, F., Saut, J.C., \textit{The Cauchy Problem for the Euler-Poisson System and Derivation of the Zakharov-Kuznetsov Equation}, \textit{Chapter 10 in Studies in Phase Space Analysis with Applications to PDEs}, Progress in Nonlinear Differential Equations and Their Applications {\bf {84}} Birkhäuser (2013), 181-213.

\bibitem{LP1} Linares, F., Pastor, A., \textit{Well-posedness for the two-dimensional modified Zakharov-–Kuznetsov equation.} SIAM J. Math. Anal. {\bf 41} (2009), 1323-1339.

%\bibitem{LP2} Linares, F., Pastor, A., \textit{Local and global well-posedness for the 2D generalized Zakharov-Kuztnesov equation}, J. Funct. Anal.{\bf  260} (2011), 1060-1085.

%\bibitem{LPS} Linares, F., Pastor, A., Saut, J.C., \textit{Well-Posedness for the ZK Equation in a Cylinder and on the Background of a KdV Soliton}, Comm. Partial Diff. Equations, {\bf 35} (2010), 1674-1689.

\bibitem{LS2009} Linares, F., Saut, J.C., \textit{The Cauchy problem for the 3D Zakharov-Kuznetsov equation}, Discrete Contin. Dyn. Syst. {\bf{24}} (2009), No 2, 547-565.

\bibitem{MP2014} Molinet, L., Pilod, D. \textit{Bilinear Strichartz estimates for the Zakharov-Kuznetsov equations and applications}, Ann. Inst. H. Poincaré Anal. Non Linéaire {\bf {32}} (2015), no. 2, 347-371.

%\bibitem{N} Nahas, J., \textit{A decay property of solutions to the k-generalized KdV equation}, Adv.Differential Equations {\bf 17} (2012), no. 9/10, 833-858.

\bibitem{NP2009} Nahas, J., Ponce, G., \textit{On the persistent properties of solutions to semi-linear Schr\"odinger
equation}, Comm. P.D.E. {\bf 34} (2009) , 1208-1227.

%\bibitem{P} Panthee, M., \textit{A note on the unique continuation property for Zakharov-–Kuznetsov equation}, Nonlinear Analysis {\bf 59}(2004), 425-438.

\bibitem{RV2012} Ribaud, F., Vento, S., \textit{Well-posedness results for the three-dimensional Zakharov-Kuznetsov equation}, SIAM J. Math. Anal. {\bf{44}} (2012), 2289-2304.

\bibitem{SS} Saut, J.C., Scheurer, B., \textit{Unique continuation for some evolution equations,} J. Differential Equations {\bf 66} (1987), 118-139.

\bibitem{ZaKu} Zakharov, V. E., Kuznetsov, E. A., \textit{On three-dimensional solitons,} Soviet Phys. JETP {\bf 39}(1974), 285-286.
\end{thebibliography}
\end{document}